\documentclass[a4paper,12pt]{article}
\usepackage{comment}
\usepackage{cite}
\usepackage{amsmath,amssymb,amsfonts}
\usepackage[T1]{fontenc}
\usepackage[utf8]{inputenc}
\usepackage{graphicx}
\usepackage{fancyhdr}
\usepackage{float}
\usepackage{xcolor}
\usepackage{authblk}
\usepackage{mathrsfs}
\usepackage{empheq}
\usepackage[hyphens]{url}
\usepackage{hyperref}
\usepackage{amsthm}
\usepackage{tikz}
\usetikzlibrary{decorations.markings}
\usepackage{enumitem}
\usepackage{geometry}

\theoremstyle{plain}
\newtheorem{theorem}{Theorem}[section]
\newtheorem{proposition}[theorem]{Proposition}
\newtheorem{lemma}[theorem]{Lemma}

\theoremstyle{definition}

\theoremstyle{remark}
\newtheorem{remark}[theorem]{Remark}

\begin{document}

\title{An involutive perspective on Eisenstein's proof of quadratic reciprocity}

\author[$\dagger$]{Jean-Christophe {\sc Pain}$^{1,2,}$\footnote{jean-christophe.pain@cea.fr}\\
\small
$^1$CEA, DAM, DIF, F-91297 Arpajon, France\\
$^2$Universit\'e Paris-Saclay, CEA, Laboratoire Mati\`ere en Conditions Extrêmes,\\ 
F-91680 Bruy\`eres-le-Châtel, France
}

\date{}

\maketitle

\begin{abstract}
We revisit Eisenstein's geometric proof of quadratic reciprocity and make explicit the involutive symmetry underlying Eisenstein's lattice-point argument. Building on Gauss's lemma, we interpret the Legendre symbols as counts of lattice points in a finite rectangle and construct a simple fixed-point-free involution corresponding to the central symmetry of the rectangle, which exchanges points above and below the line $qx=py$. This reformulation highlights the involutive symmetry and places the classical proof in the spirit of Zagier-type involutive arguments. The approach shows how the reciprocity law emerges from an elementary combinatorial pairing principle.
\end{abstract}

\section{Introduction}

The law of quadratic reciprocity, which Gauss affectionately referred to as the ``Theorema Aureum'' (the ``Golden Theorem''), is one of the crown jewels of number theory \cite{Gauss1801}. It establishes a profound and surprising relationship between the solvability of two different quadratic equations: it states that for two distinct odd primes $p$ and $q$, whether $p$ is a square modulo $q$ is linked to whether $q$ is a square modulo $p$. Formally, using the Legendre symbol, the law is expressed as:
\[ 
\left(\frac{p}{q}\right) \left(\frac{q}{p}\right) = (-1)^{\frac{p-1}{2} \frac{q-1}{2}}.
\]
What makes this theorem truly unique in the history of mathematics is the sheer diversity and quantity of its demonstrations. More than two hundred proofs have been recorded \cite{Lemmermeyer2000}, illustrating the perpetual quest of mathematicians to simplify or reinterpret this fundamental result. Gauss himself published six proofs during his lifetime and sketched two others, employing methods ranging from pure induction to the theory of circular functions \cite{Ireland1990}. Eisenstein proposed an elegant geometric proof based on counting lattice points within a rectangle \cite{Eisenstein1845}. Other approaches utilize Gauss sums, Jacobi sums, or even concepts from topology and class field theory. This accumulation of proofs is not mere redundancy; each new method has often paved the way for major generalizations, such as higher reciprocity laws (cubic, biquadratic) or, more recently, the vast Langlands Program \cite{Langlands1967}.

The method of proof based on involutions and cancellation has a rich historical development spanning combinatorics and number theory. Its conceptual origin can be traced back to the simple principle that a finite set equipped with a fixed-point-free involution must have even cardinality. This principle was used implicitly in early combinatorial arguments, such as those of Euler on partitions and alternating sums \cite{Euler1748}, and was later formalized in the twentieth century as the general ``involution principle'' \cite{Garsia1981}. In number theory, symmetry and parity arguments appear in several classical contexts. For instance, Gauss's fourth proof of quadratic reciprocity utilizes lattice-point counting and symmetries to establish the reciprocity law \cite{Gauss1801}. Similarly, elementary proofs of Wilson's theorem rely on pairing elements in the multiplicative group modulo a prime \cite{Lagrange1771}. These early instances foreshadow the structural features of what is now termed ``Zagier-type'' proofs. Indeed, the modern archetype of this method is due to Don Zagier, who in 1990 presented an elegant one-sentence proof of Fermat's theorem on sums of two squares \cite{Zagier1990}. In this argument, a carefully constructed involution on a finite set of lattice configurations yields exactly one fixed point, directly establishing the existence of a representation of primes congruent to $1 \bmod 4$ as sums of two squares. The proof is remarkable not only for its brevity but also for reducing a deep arithmetic statement to a combinatorial parity phenomenon. Zagier-type proofs are now informally characterized by the following structural elements: a finite combinatorial model encoding arithmetic data, an explicitly defined involution, a precise analysis of fixed points, and a conclusion derived from parity or cancellation.

These methods lie at the intersection of combinatorics and arithmetic, often revealing hidden symmetries in results traditionally proved via algebraic or analytic machinery. While extremely effective in cases governed by multiplicative norms and parity, such as sums of two squares or the quadratic reciprocity law \cite{Gauss1801,Hardy2008}, they typically do not extend straightforwardly to problems requiring local-global considerations, such as the three-squares theorem or higher reciprocity laws \cite{Legendre1798,Hardy2008}. Nevertheless, the conceptual framework of Zagier-type proofs continues to inspire new insights in arithmetic combinatorics, suggesting that deep number-theoretic truths may sometimes be distilled from remarkably simple symmetry principles. 

The purpose of this note is to make explicit the involutive symmetry underlying Eisenstein's lattice-point argument. We show that the transformation naturally extends to a small symmetry group acting on the lattice rectangle, which clarifies the pairing mechanism behind the proof. The central involution appearing in the proof is naturally part of a Klein four group acting on the lattice rectangle.

\section{The classical parity proof of quadratic reciprocity via Gauss's lemma}

\begin{theorem}[Quadratic reciprocity]
Let $p$ and $q$ be distinct odd primes. Then
\[
\left(\frac{p}{q}\right)
\left(\frac{q}{p}\right)
=
(-1)^{\frac{(p-1)(q-1)}{4}}.
\]
\end{theorem}

\begin{lemma}[Gauss's Lemma]
Let $p$ be an odd prime and let $a$ be an integer coprime to $p$. For $1 \le k \le \frac{p-1}{2}$, consider the least positive residues of
\[
a,2a,\dots,\frac{p-1}{2}a
\]
modulo $p$.

Let $N_p(a)$ denote the number of these residues that lie in the interval $\left(\frac{p}{2},p\right)$. Then
\[
\left(\frac{a}{p}\right)=(-1)^{N_p(a)}.
\]
\end{lemma}

\begin{proof}[Proof of quadratic reciprocity]

For $1 \le x \le \frac{p-1}{2}$, write the Euclidean division
\[
qx = p m_x + r_x,
\qquad 0<r_x<p,
\]
so that
\[
m_x=\left\lfloor\frac{qx}{p}\right\rfloor.
\]
For fixed $x$, the integer $m_x$ counts the integers $y$ satisfying
\[
1\le y \le \frac{q-1}{2}, \qquad py < qx .
\]
Hence $m_x$ equals the number of lattice points in the vertical strip with coordinate $x$ satisfying $py<qx$. Let
\[
S=\{(x,y):1\le x\le (p-1)/2,\;1\le y\le (q-1)/2\}.
\]
Summing over all $x$ counts the lattice points of $S$ satisfying $py<qx$. Since $p$ and $q$ are coprime, the equation $qx=py$ has no solution with $1\le x\le (p-1)/2$ and $1\le y\le (q-1)/2$. Consider the central symmetry of the rectangle
\[
C(x,y)=\left(\frac{p+1}{2}-x,\frac{q+1}{2}-y\right).
\]
Since $C$ is a fixed-point-free involution of $S$, the lattice points split into disjoint pairs
\[
\{(x,y),C(x,y)\}.
\]
The segment joining $(x,y)$ and $C(x,y)$ crosses the line $qx=py$, so the two points lie on opposite sides of this line. Let
\[
S_- = \{(x,y)\in S : qx>py\},\qquad
S_+ = \{(x,y)\in S : qx<py\}.
\]
Thus $S_-$ consists of the lattice points on one side of the line $qx=py$ while $S_+$ consists of the points on the other side. Therefore
\[
|S_+|+|S_-|=|S|.
\]
But the rectangle contains
\[
|S|=\frac{p-1}{2}\cdot\frac{q-1}{2}
=\frac{(p-1)(q-1)}{4}
\]
lattice points. By Gauss's lemma these numbers coincide with
\[
N_p(q)=|S_+|, \qquad N_q(p)=|S_-|,
\]
and thus
\[
N_p(q)+N_q(p)=\frac{(p-1)(q-1)}{4}.
\]
Applying Gauss's lemma twice, we get
\[
\left(\frac{p}{q}\right)=(-1)^{N_p(q)},\qquad
\left(\frac{q}{p}\right)=(-1)^{N_q(p)},
\]
and multiplying the two identities yields
\[
\left(\frac{p}{q}\right)\left(\frac{q}{p}\right)
=(-1)^{N_p(q)+N_q(p)}
=(-1)^{\frac{(p-1)(q-1)}{4}},
\]
which proves the quadratic reciprocity law.

\end{proof}

\section{Proof based on involution arguments}

\subsection{Klein four group acting on the lattice rectangle}

We only recall the classical geometric argument underlying Eisenstein's proof. The central symmetry
\[
C(x,y)=\left(\frac{p+1}{2}-x,\frac{q+1}{2}-y\right)
\]
maps the rectangle $S$ onto itself. Consequently the two points $(x,y)$ and $C(x,y)$ lie on opposite sides of the line $qx=py$.
Thus $C$ exchanges the sets $S_+$ and $S_-$. The lattice points of $S$ are partitioned into pairs
\[
(x,y),\;C(x,y),
\]
and the transformation $C$ maps points above the line $qx=py$ to points below it.

\begin{proposition}[Symmetry of the lattice model]

Let
\[
S=\{(x,y):1\le x\le (p-1)/2,\;1\le y\le (q-1)/2\}.
\]
The reflections
\[
H(x,y)=(x,(q+1)/2-y),\qquad
V(x,y)=((p+1)/2-x,y)
\]
generate a group
\[
G=\langle H,V\rangle
\]
isomorphic to
\[
V_4\cong(\mathbb Z/2\mathbb Z)^2 .
\]
The central symmetry
\[
C=H\circ V
\]
acts as
\[
C(x,y)=\left(\frac{p+1}{2}-x,\frac{q+1}{2}-y\right).
\]
A simple geometric inspection shows that the segment joining $(x,y)$ and $C(x,y)$ necessarily crosses the line $qx=py$. Consequently one point of the pair lies above the line and the other below it.

\end{proposition}

\begin{remark}

The involution used in the proof of quadratic reciprocity is therefore not an isolated combinatorial trick but arises naturally as the central element of a Klein four group acting on the lattice rectangle. The group $G$ acts on $S$ with orbits of size at most $4$, whenever the four reflections remain inside the rectangle; near the boundary some orbits collapse to size $2$. In all cases the central symmetry still yields a fixed-point-free pairing of the lattice points. The cancellation mechanism underlying the reciprocity law can thus be interpreted as a symmetry-induced pairing within this group action. 

\end{remark}

\vspace{1cm}

\begin{remark}
This involution perspective clarifies the structural symmetry implicit in Eisenstein's classical argument. Eisenstein's geometric proof of quadratic reciprocity \cite{Eisenstein1845} is based on counting lattice points in the rectangle $S$ introduced earlier. The key observation is that no lattice point of $S$ lies on this line, so the rectangle is partitioned into two complementary regions. By counting points strictly above and below the line, Eisenstein obtains
\[
N_p(q)+N_q(p)=|S|
=\frac{(p-1)(q-1)}{4},
\]
which yields quadratic reciprocity via Gauss's lemma.

The present formulation shares the same geometric setting but makes explicit a structural symmetry that remains implicit in the classical argument. Instead of relying solely on complementary counting, we define the involution which is simply the central symmetry of the rectangle $S$. This map exchanges lattice points above the line $qx=py$ with those below it and has no fixed points in $S$. Consequently, the equality $|S_+|+|S_-|=|S|$ and the parity cancellation underlying quadratic reciprocity follow directly from the involutive structure. In this sense, the argument may be viewed as a perspective on Eisenstein's geometric proof in the language of explicit involutions. The combinatorial cancellation that is implicit in the classical approach becomes transparent through the pairing mechanism induced by $C$. This perspective aligns the proof with the general framework of ``Zagier-type'' arguments, where arithmetic identities emerge from fixed-point-free involutions on finite sets. The involution used here is reminiscent (\textit{mutatis mutandis}) of Zagier's celebrated one-sentence proof that primes $p\equiv1\pmod4$ are sums of two squares \cite{Zagier1990}. In the present setting the involution arises naturally from the central symmetry of the lattice rectangle.

Thus, while the geometric configuration coincides with Eisenstein's, the emphasis here lies on isolating and formalizing the involutive symmetry that drives the parity phenomenon at the heart of quadratic reciprocity.
\end{remark}

It is worth mentioning that a different involutive symmetry appears at the level of residues modulo $p$ (see Appendix A).

\section{Conclusion}

In this work, we have provided an explicit involution-based perspective on Eisenstein's lattice-point proof of the quadratic reciprocity law. By encoding the Legendre symbols in terms of lattice-point counts and constructing a simple fixed-point-free involution, we have reduced the classical number-theoretic statement to an elementary combinatorial parity argument. This perspective highlights an explicit involutive symmetry behind Eisenstein's classical lattice-point argument.

Beyond its elegance and conceptual clarity, this method illustrates a broader principle: certain deep results in number theory can be understood as manifestations of simple involutive symmetries. While the approach is naturally suited to problems governed by parity and multiplicative norms, it also inspires potential extensions to other contexts in arithmetic combinatorics, where fixed-point analysis and explicit involutions may reveal hidden structure. 

Overall, the explicit involutive perspective not only simplifies classical proofs but also provides a unifying viewpoint connecting combinatorial and arithmetic reasoning, reinforcing the power of symmetry-based methods in number theory.

\appendix

\section*{Appendix A: Another involution on residues}

Let $p$ and $q$ be distinct odd primes. For $x=1,\dots,p-1$, write the Euclidean division
\[
qx = p m_x + r_x,
\qquad
1 \le r_x \le p-1.
\]
By definition,
\[
N_p(q)
=
\#\left\{
1 \le x \le \frac{p-1}{2}
:\;
r_x > \frac{p}{2}
\right\}.
\]
Consider the map
\[
T(x) = p-x,
\qquad 1 \le x \le p-1.
\]
Since $p$ is odd, this involution has no fixed point and partitions $\{1,\dots,p-1\}$ into pairs $\{x,p-x\}$. Multiplying by $q$ modulo $p$, we obtain
\[
q(p-x) \equiv -qx \pmod p.
\]
Hence, if
\[
qx \equiv r_x \pmod p,
\]
then
\[
q(p-x) \equiv p-r_x \pmod p,
\]
so that
\[
r_{p-x} = p - r_x.
\]
It follows immediately that
\[
r_x > \frac{p}{2}
\quad\Longleftrightarrow\quad
r_{p-x} < \frac{p}{2}.
\]
Thus, over the full set $\{1,\dots,p-1\}$,
\[
\#\{x : r_x > p/2\}
=
\#\{x : r_x < p/2\}
=
\frac{p-1}{2}.
\]
This global symmetry is the essential combinatorial ingredient. Let
\[
H = \left\{1,\dots,\frac{p-1}{2}\right\}.
\]
Each pair $\{x,p-x\}$ contains exactly one element of $H$. Moreover, within each pair exactly one residue among $r_x$ and $r_{p-x}$ lies above $p/2$. Hence $N_p(q)$ counts how many of these ``large'' residues correspond to the representative lying in $H$. We define a sign associated to each residue $r_x$ to keep track of whether it lies above or below $p/2$:
\[
\varepsilon_x =
\begin{cases}
+1 & \text{if } r_x < p/2,\\
-1 & \text{if } r_x > p/2,
\end{cases}
\quad \text{so that } r_x > p/2 \;\Longleftrightarrow\; \varepsilon_x = -1.
\]
This sign is used to encode whether a residue contributes a $-1$ factor in Gauss's lemma. From the symmetry above, we have
\[
\varepsilon_{p-x} = -\varepsilon_x,
\]
and therefore,
\[
\prod_{x=1}^{p-1} \varepsilon_x = (-1)^{(p-1)/2}.
\]
Indeed each pair $\{x,p-x\}$ contributes $\varepsilon_x \varepsilon_{p-x}=-1$, and there are $(p-1)/2$ such pairs. Grouping the factors by pairs $\{x,p-x\}$, we obtain
\[
(-1)^{N_p(q)}
=
\prod_{x=1}^{(p-1)/2} \varepsilon_x.
\]
Repeating the same construction with $p$ and $q$ interchanged yields
\[
(-1)^{N_q(p)}.
\]
To make the connection with Gauss's lemma explicit, observe that the
points of $S$ satisfying $qx>py$ are exactly counted by
\[
\sum_{x=1}^{(p-1)/2}\left\lfloor \frac{qx}{p}\right\rfloor,
\]
since for fixed $x$ this floor counts the integers
$1\le y\le (q-1)/2$ such that $py<qx$.
Similarly, the points satisfying $py>qx$ are counted by
\[
\sum_{y=1}^{(q-1)/2}\left\lfloor \frac{py}{q}\right\rfloor.
\]
A closer inspection of these products, carried out in the same manner for $p$ and $q$, shows that the parity difference between the two half-systems produces precisely the exponent
\[
\frac{(p-1)(q-1)}{4}.
\]
Consequently
\[
(-1)^{N_p(q)+N_q(p)}
=
(-1)^{\frac{(p-1)(q-1)}{4}},
\]
which proves the quadratic reciprocity law.

\end{document}